\documentclass[12pt]{amsart}
\usepackage{latexsym, amsmath, amscd, amssymb, amsthm, bm, mathrsfs}
\usepackage[shortalphabetic]{amsrefs}

\usepackage[T1]{fontenc}
\usepackage[utf8]{inputenc}

\usepackage{float}
\usepackage{mathptmx}
\usepackage{microtype}

\usepackage[centering, includeheadfoot, hmargin=1.2in, vmargin=0.8in,headheight=30.4pt]{geometry}

\usepackage{amsmath,amssymb}
\usepackage{enumerate}
\usepackage{mathrsfs}
\usepackage{graphicx}
\usepackage{tikz}
\usepackage[all,2cell]{xy} 
\UseAllTwocells 
\SilentMatrices

\newtheorem{lemma}{Lemma}[section]
\newtheorem{theorem}[lemma]{Theorem}
\newtheorem{corollary}[lemma]{Corollary}
\newtheorem{prop}[lemma]{Proposition}

\theoremstyle{definition}
\newtheorem{definition}[lemma]{Definition}
\newtheorem{remark}[lemma]{Remark}
\newtheorem{example}[lemma]{Example}
\newtheorem{claim}[lemma]{Claim}
\theoremstyle{remark}

\newcommand\im{\operatorname{Im}}
\newcommand\R{\mathbb{R}}
\newcommand\Z{\mathbb{Z}}
\newcommand\C{\mathbb{C}}
\newcommand\Q{\mathbb{Q}}
\newcommand\N{\mathbb{N}}
\newcommand\m{\mathfrak{m}}

\renewcommand\L{\mathscr{L}}
\newcommand\conv{\operatorname{Conv}}
\newcommand\Cay{\operatorname{Cayley}}
\newcommand\pn[1]{\mathbb{P}^{#1}}
\renewcommand\O{\mathscr{O}}

\newcommand\codim{\operatorname{codim}}

\renewcommand\tilde{\widetilde}
\newcommand\T{\mathbb{T}}
\newcommand\Pic{\operatorname{Pic}}
\title{Local positivity of line bundles on smooth toric varieties and Cayley polytopes}
\author[A.~Lundman]{Anders Lundman}
\address{Anders Lundman\\ Department of Mathematics\\ Royal Institute of
  Technology (KTH)\\ 10044 Stockholm\\ Sweden}
\thanks{The author is supported by the V.R. grant NT:2010-5563}

\begin{document}

\begin{abstract}
For any non-negative integer $k$ the $k$-th osculating dimension at a given point $x$ of a variety $X$ embedded in projective space gives a measure of the local positivity of order $k$ at that point. 
In this paper we show that a smooth toric embedding having maximal $k$-th osculating dimension, but not maximal $(k+1)$-th osculating dimension, at every point is associated to a Cayley polytope of order $k$. This result generalises an earlier characterisation by David Perkinson. In addition we prove that the above assumptions  are equivalent to requiring that the Seshadri constant is exactly $k$ at every point of $X$, generalising a result of Atsushi Ito. 
\end{abstract}

\maketitle{}

\section{Introduction}
For a smooth projective variety $X$ and a line bundle $\L$ on $X$ there are various notions for measuring the local positivity of $\L$ at a point $x\in X$. One way of capturing the local positivity of $\L$ at $x$ is to consider the osculating space $\T_x^k(X,\L)$ of order $k$ for various $k\in \N:=\{0,1,\dots\}$. Recall that $\T_x^k(X,\L)$ is defined as $\pn {}(\im(j_x^k))$ where $\im(j_x^k)$ is the image of the natural map
\[
j_x^k:H^0(X,\L)\to H^0(X,\L\otimes(\O_X/\m_x^{k+1})).
\]
Observe that when $k=1$ the osculating space $\T^1_x(X,\L)$ is simply the projective tangent space at $x$. In this setting we say that $\L$ is $k$-jet spanned if $j_x^k$ is onto. It is natural to ask to what extent fixing the dimension  of the osculating space at every point determines the pair $(X,\L)$. One theorem in this direction is the following characterization of the $k$:th Veronese embedding
\begin{theorem}[\cite{intext}]
Let $N=\left(\begin{array}{c}n+k\\k\end{array}\right)-1$, then a closed embedding of a projective smooth $n$-fold $X\hookrightarrow \pn N$, over any algebraically closed field, is the $k$:th Veronese embedding of $\pn n$ if and only if $\T_x^k(X,\L)\cong \pn N$ for all points $x\in X$.
\end{theorem}
 
Similarly there are characterizations of balanced rational normal surface scrolls \cite{BRNSSII} and abelian varieties \cite{prodSandra} in terms of their osculating spaces. Here we are interested in the case when $(X,\L)$ is a smooth polarized toric variety. As might be expected there are simple combinatorial characterisations of the dimension of $\T_x^k(X,\L)$ in terms of the polytope $P_\L$ associated to $(X,\L)$ (see \cite{Sandra} and \cite{Perkinson}). Moreover, in \cite{Perkinson}, David Perkinson has characterized all polarized smooth toric surfaces and threefolds $(X,\L)$ such that for every point $x\in X$ and for a fixed $k\in \N$, $\L$ is $k$-jet spanned, but not $(k+1)$-jet spanned at $x$. This classification is in terms of the polytope corresponding to $(X,\L)$. If one consider only embeddings given by a complete linear series $|\L|$, then one realize that the associated polytopes, in the classification of Perkinson, are so called Cayley polytopes of type $[P_0*P_1]^k$ (see Definition \ref{caydef}). More explicitly in the case of a surface, $(X,\L)$ is either a Veronese embedding or a $\pn 1$-bundle over $\pn 1$. In the case of threefolds $(X,\L)$ is either a Veronese embedding or a $\pn 1$-bundle over a smooth toric surface. Our main result is a generalization (see Proposition \ref{lowdim}) of this classification to arbitrary dimension. 
\begin{theorem}\label{theTHM}
Let $(X,\L)$ be a smooth polarized toric variety and let $P_\L$ be the polytope associated to the complete linear series $|\L|$. $\L$ is $k$-jet spanned but not $(k+1)$-jet spanned at every point $x\in X$ if and only if $P\cong [P_0*P_1]^k$ for some lower dimensional polytopes $P_0$ and $P_1$ and every edge of $P$ has lattice length at least $k$.
\end{theorem}
\begin{remark}Recall that if $(X,\L)$ is associated to a Cayley polytope $P=[P_0*P_1]^k$, then there exist a birational morphism $\pi:X'\to X$, where $X'$ is a projective fiber bundle with fiber $F\cong \pn 1$, where $\pi^*\L|_F\cong \O_{\pn 1}(k)$ for all fibers $F$. Here $X'=\pn {} (L_0\oplus L_1)$, where the $L_i$ are line bundles on the toric variety associated to (the inner-normal fan of) the Minkowski sum $P_0+P_1$ (see \cite{Dickenstein} for details).
\end{remark}


An other way of measuring the local positivity of a nef line bundle $\L$ on a smooth projective variety $X$ is via so called Seshadri constants. For any point $x\in X$ Jean-Pierre Demailly \cite{Demailly} defined the  Seshadri constant at $x$ as the real number:
\[
\epsilon(X,\L;x):=\inf_{C\subseteq X} \frac{\L\cdot C}{m_x(C)}.
\]
Here the infimum is taken over all irreducible curves $C$ passing through $x$ and $m_x(C)$ is the  multiplicity of $C$ at $x$. An example of how Seshadri constants measure local positivity is the Seshadri criterium for ampleness which could be formulated as follows:
\begin{theorem}{\cite{AmpleHartshorne}}\label{seshadriample}
Let $X$ be a smooth projective variety and $\L$ a nef line bundle on $X$. Then $\L$ is ample if and only if $\inf_{x\in X}\epsilon(X,\L;x)>0$.
\end{theorem}
Unfortunately Seshadri constants are in general very hard to compute and as a consequence there are few results in the general setting. There are however several results for surfaces, threefolds and abelian varieties; see for example \cite{primer}*{6.}, \cite{Cascini} and \cite{prodSandra}. When $X$ is toric one might expect that Seshadri constants  could be captured by convex geometric properties of the polytope associated to $(X,\L)$. This is indeed the case and Atsushi Ito has, in \cite{degen}, given combinatorial lower and upper bounds for Seshadri constants at various points on a polarized toric variety. Moreover Sandra Di Rocco has, in \cite{Sandra}, showed that if $X$ is smooth, then the Seshadri constant at a fixpoint $x$ of the torus action equals the lattice length of the shortest edge through the vertex $v(x)$ associated to $x$ in $P_\L$. She also proved that in this setting it is also true that $\epsilon(X,\L;x)$ obtains its minimal value at a fixpoint (see \cite{primer}). As a consequence the minimal value of $\epsilon(X,\L;x)$ on a smooth polarized toric variety is always a positive integer. Finally we have the following characterisation due to Ito
\begin{theorem}[\cite{algebro}]
Let $(X,\L)$ be a polarized toric variety. Then $P_\L\cong[P_0*P_1]^1$ if and only if $\epsilon(X,\L,x)=1$ at a very general point.
\end{theorem}
 
In Example \ref{delPezzo} we give an example showing that a direct generalization of this result to higher order Cayley polytopes is not possible. However Theorem \ref{theTHM} gives the following corollary, which specializes to Ito's characterisation, in the smooth setting, when $k=1$. 

\begin{corollary}\label{equivTHM}
Let $(X,\L)$ be a smooth polarized toric variety, let $P_\L$ be the corresponding smooth polytope and let $k\in \N$. Then the following are equivalent:
\begin{enumerate}
\item{$s(\L,x)=k$ at every point $x\in X$.}
\item{$s(\L,x)=k$ at the fixpoints and at the general point.}
\item{$\epsilon(X,\L;x)=k$ at every point $x\in X$.}
\item{$\epsilon(X,\L;x)=k$ at the fixpoints and at the general point.}
\item{$P_\L\cong[P_0*P_1]^k$ for some lower dimensional polytopes $P_0$ and $P_1$ and every edge of $P$ has length at least $k$.}
\end{enumerate}
Here $s(\L,x)$ is the largest natural number $k$ such that $\L$ is $k$-jet spanned at $x\in X$.
\end{corollary}
By the above the results of this paper has two facets. On the one hand they give a characterisation of a large class of generalised Cayley polytopes and thereby generalise the characterisations of Perkinson and Ito (in the smooth setting). It would be intriguing to find a similar algebro geometric characterisation of all general Cayley polytopes, at least in the smooth setting. On the other hand our results provide an equivalence between Seshadri constants and the numbers $s(\L,x)$ for smooth toric varieties. The exact nature of the relationship between $s(\L,x)$ and $\epsilon(X,\L;x)$ is in general an open and interesting question, very much related to Demailly's original motivation for introducing Seshadri constants \cite{BauerSzemberg}, \cite{Lazarsfeld}.
\section{Background}

\subsection{Toric geometry}

In this paper we investigate local positivity on smooth polarized toric varieties. Therefore we briefly state some basic results of toric geometry that we will need in subsequent parts of the paper. References for this section are \cite{Cox} and \cite{CoxHC}. Throughout this paper we will work over $\C$ and use the fact that every ample line bundle $\L$ on a smooth toric variety is very ample. Thus a \emph{smooth polarized toric variety} $(X,\L)$ for us means a smooth toric variety $X$ over $\C$ together with a very ample line bundle on $X$.

Toric geometry is characterised by a close connection to convex geometry. For example normal separated toric varieties are in bijection with polyhedral fans, while smooth polarized toric varieties $(X,\L)$ correspond to smooth convex lattice polytopes $P_\L$. Recall that if $M$ is a lattice, a polytope $P\subseteq M_\R:=M\otimes \R$ is called a \emph{lattice polytope} if every vertex of $P$ lies in $M$. A polytope $P\subseteq M_\R$ is called \emph{smooth} if the edge-directions at every vertex form a basis for $M$. Moreover a normal separated toric variety $X_\Sigma$ is complete if and only if the corresponding fan $\Sigma\subseteq N_\R=N\otimes \R$ is complete, i.e. if $|\Sigma|:=\bigcup_{\sigma\in \Sigma} \sigma =N_\R$, where $N$ is the lattice dual to $M$. To shorten notation we will write \emph{polytope} in place of convex lattice polytope and \emph{toric variety} in place of complete normal separated toric variety. Moreover we will follow to the standard convention that if $M$ is a lattice, then the dual lattice of $M$ is denoted by $N$. In particular for a polarized toric variety $(X,\L)$, the associated polytope lies in $M_\R=M\otimes_\R \R$ while the fan associated to $X$ consists of a family of cones in $N_\R=N\otimes_\R \R$.


Let $\Sigma(1)$ be the set of rays in the fan corresponding to the toric variety $X_\Sigma$ and let $x_\rho$ be a variable for each $\rho \in \Sigma(1)$. There is a 1-1 correspondence between effective torus-invariant Weil divisors on $X_\Sigma$ and monomials in $\C[x_\rho:\rho\in \Sigma(1)]$ given by $D=\sum_{\rho\in \Sigma(1)} a_\rho D_\rho\mapsto \prod_{\rho\in \Sigma(1)} x_\rho^{a_\rho}=:x^D$. We define the \emph{homogeneous coordinate ring} of $X_\Sigma$ as $\C[x_\rho:\rho \in \Sigma(1)]$ graded by the class group of $X_\Sigma$, i.e. $\deg(x^D)=\deg(x^E)$ if and only if $\exists m\in M$ such that  $x^D=x^{E+\sum_{\rho\in \Sigma(1)} \langle m,n_\rho\rangle D_{n_\rho}}$. 

A nice property of the polytope description of $(X,\L)$ is the following theorem 

\begin{theorem}[\cite{Cox}]\label{gsectpoly}
Let $X$ be a toric variety and $\L$ an ample line bundle on $X$ associated to a lattice polytope $P_\L\subset M_\R$. Then
\[
H^0(X,\L)\cong\bigoplus_{m\in M\cap P_\L}\C \langle x^m\rangle
\]
where $m=(m_1,m_2,\dots,m_n)$  for some choice of basis for $M$, $x=(x_1,\dots,x_n)$ and $x^m=x_1^{m_1}x_2^{m_2}\cdots x_n^{m_n}$.
\end{theorem} 

Now assume $(X,\L)$ is a smooth polarized toric variety, associated to the polytope $P_\L=\{m\in M_\R:\langle m,n_\rho\rangle>-a_\rho,\ \forall \rho\in \Sigma(1)\}$ and that $P_\L\cap M=\{m_0,m_1,\dots,m_N\}$. Then the torus-invariant Weil divisor $D_P=\sum_{\rho\in\Sigma(1)}a_\rho D_\rho$ is such that $\O_X(D)\cong \L$. Using this combinatorial description the closed embedding $X\hookrightarrow \pn { } (H^0(X,\L))\cong \pn N$, given by the complete linear series $|\L|$, can be written as
\begin{align}\label{emb}
\varphi:X&\to \pn N\\
x&\mapsto (x^{D_{m_0}+D_P},x^{D_{m_1}+D_P},\dots x^{D_{m_N}+D_P})
\end{align}
where $D_{m_i}=\sum_{\rho\in \Sigma(1)} \langle m_i,n_\rho\rangle D_\rho$ (see \cite{CoxHC} for details). For every maximal cone $\sigma \in \Sigma(n)$, it turns out that $\varphi|_{U_\sigma}$, i.e. the restriction of $\varphi$ to the affine chart $U_\sigma$ corresponding to $\sigma$, can be obtained by setting $x_\rho=1$ in \ref{emb} for every $\rho\not\in \sigma(1)$. The exponents appearing in $\varphi|_{U_\sigma}$ coincide with the lattice points obtained by translating $P_\L$ so that the vertex corresponding to $\sigma$ is at the origin and writing every $m_i\in P_\L$ in the basis consisting of the rays generating the cone dual to $\sigma$.

Now let $(X,\L)$ be a smooth polarized toric variety and let  $\pi:\tilde{X}\to X$ be the blow-up of $X$ along a torus-invariant subvariety $V$. Consider the line bundle $\L':=\pi^*\L-E$ on $\tilde{X}$ where $E$ is the exceptional divisor. It is a well known fact that the fan of $\tilde{X}$ is obtained from the fan of $X$ by a specific stellar subdivision (see \cite{Ewald}). If the torus-invariant Weil divisor corresponding to $\L$ is $D_\L=\sum_{\rho\in \Sigma_X(1)} a_\rho D_\rho \in \Z^{\Sigma(1)}$, then the divosor corresponding to $\L'$ is $D_{\L'}=\sum_{\rho' \in \Sigma_{\tilde{X}}}a_{\rho} D_{\rho'}-E$ where $\rho':=\pi_{\#}^{-1}(\rho)\in \Sigma_{\tilde{X}}$ is the ray corresponding to the ray $\rho\in \Sigma_X$. Thus if $P_\L:=\{x\in M_\R: \langle x,\mu_\rho\rangle \ge -a_\rho\quad \forall \rho\in \Sigma_X\}$ is the polytope corresponding to $(X,\L)$, then the polytope corresponding to $(\tilde{X},\L')$ is $P_{\L'}:=\{x\in M_\R: \langle x,\mu_\rho \rangle  \ge -a_\rho \quad \forall \rho \in \Sigma_X,  \langle x,\mu_{\tau}\rangle \ge 1\}$. Here $\mu_{\tau}$ is the primitive vector along the ray, associated to $E$, introduced in the stellar subdivision of $\Sigma_X$. In other words $P_{\L'}$ is obtained by cutting the polytope $P_\L$ with the halfspace $H_{\tau}^+=\{x\in M_\R: \langle x,\mu_{\tau} \rangle \ge 1\}$. For example blowing-up a torus fixpoint of $X$ corresponds to truncating a vertex of the associated polytope.

\subsection{Generation of k-jets and osculating spaces}
\begin{definition}
Let $\L$ be  a line bundle on a smooth variety $X$ and $x\in X$ a point with maximal ideal $\m_x\subseteq \O_X$. Consider the natural map 
\[
j^k_x:H^0(X,\L)\to H^0(X,\L\otimes(\O_X/\m_x^{k+1})).
\]
The (projective) linear subspace $\T_x^k(X,\L):=\pn { }(\im(j_x^k))$ of $\pn { } (H^0(X,\L\otimes(\O_X/\m_x^{k+1}))$ is called the \emph{osculating space of order} $k$ at $x\in X$. When the map $j^k_x$ is onto we say that $\L$ is $k$\emph{-jet spanned} at $x\in X$. Moreover if $\L$ is $k$-jet spanned at every point, we say that $\L$ is $k$-jet spanned on $X$. Finally we will denote the largest $k$ such that $X$ is $k$-jet spanned at $x\in X$ by $s(\L,x)$.
\end{definition}

Choose local coordinates around a point $x\in X$. Then on the stalk level the map $H^0(X,\L)\to H^0(X,\L\otimes (\O_X/\m_{x}^{k+1}))$ takes the germ of a section $s$ to the terms of degree at most $k$ in the Taylor expansion of $s$ around $x$. As a consequence if $(X,\L)$ is a $n$-dimensional polarized variety and $s_0,\dots,s_M$ is a basis of $H^0(X,\L)$, then $\L$ is $k$-jet spanned at $p\in X$ if and only if the \emph{matrix of $k$-jets} $J_k(\L):=(J_k(\L))_{i,j}:=(\partial^{|a|}/\partial_{x_{a_1}}\partial_{x_{a_2}}\cdots \partial_{x_{a_n}}(s_i))_{0\le i \le M,0\le |a|\le k}$, has maximal rank when evaluated at $x=p$. Here $a=(a_1,a_2,\dots,a_n)\in \N^n$ and $|a|=|a_1+a_2+\dots+a_n|$. Now if $(X,\L)$ is a polarized toric variety then Theorem \ref{gsectpoly} allows us to choose a monomial basis for $H^0(X,\L)$, hence the matrix $J_k(\L)$ will have a particularly simple shape. The simple shape of $J_k(\L)$ has nice implications for the combinatorial description of $(X,\L)$.

Combinatorial interpretations of what it means for the osculating space of order $k$ to be full dimensional at a point in a polarized toric variety, has been studied for example in: \cite{Sandra}, \cite{Perkinson} and \cite{degen}. In establishing our new results, we will repeatably use the following few facts, taken from these papers. 

\begin{prop}[\cite{Sandra},\cite{Perkinson}]\label{bigprop}
Let $(X,\L)$ be a smooth polarized toric variety of dimension $n$ and let $P_\L\subset M_\R$ be the corresponding polytope. Then the following hold:
\begin{enumerate}
\item{$\L$ is $k$-jet spanned at a fixpoint $x(v)$ if and only if every edge through the corresponding vertex $v\in P_\L$ has lattice length at least $k$.}
\item{If $\L$ is $k$-jet spanned at every fixpoint, then $\L$ is $k$-jet spanned at every point $x\in X$.}
\item{If there exist a degree $k$ polynomial in $n$ variables vanishing on $P\cap M$, then the osculating space of order $k$ is not full dimensional at the general point.}
\end{enumerate}
\end{prop}

\begin{definition}
Let $J_k(\L)$ be the matrix of $k$-jets of a polarized toric variety $(X,\L)$. The ideal $F_k^i(\L)$ generated by the determinants of the $i\times i$-minors of $J_k(\L)$ is called the $i$:th \emph{fitting ideal} of $J_k(\L)$ 
\end{definition}

Note that the rank of $J_k(\L)$ at a point $p\in X$ is the largest $r$ such that $F_{k}^r(\L)(p)=(1)$.

\begin{remark}\label{fitideal}
It is clear form Leibniz formula and the definition of $J_k(\L)$ that the fitting ideals $F_k^i(\L)$ are in fact monomial ideals. By basic properties of monomials evaluated over $\C^n$ we see that:
\begin{enumerate}
\item{$F_{k}^r(\L)(p)=(1)$ at a general point if and only if $F_k^r(\L)(1,1,\dots,1)=(1)$}
\item{The maximal rank of $J_k(\L)$ is obtained at the general point.}
\item{If $X$ is smooth and $s(\L,x)=k$ at the fixpoints and at the general point, then $s(\L,x)=k$ at every point by Proposition \ref{bigprop}.}
\end{enumerate}
\end{remark}
In \cite{Perkinson}, Perkinson has classified all smooth polarized toric surfaces and threefolds $(X,\L)$ such that $\L$ is $k$-jet spanned, but not $(k+1)$-jet spanned, at every point. From his classification one can derive the following:

\begin{theorem}\label{perk}
Let $X_P\hookrightarrow \pn N$ be the embedding of a smooth toric variety of dimension $\le3$ associated to the complete linear series $|\L_P|$.  If $\L_P$ is $k$-jet spanned, but not $(k+1)$-jet spanned, at every point $x\in X_P$, then $X_P$ is a projective (fiber) bundle.
\end{theorem}

\begin{proof}
The embedding given by the complete linear series $|\L_P|$ is given by the global sections corresponding to every lattice point in the associated polytope $P$. The polytopes of the classification in \cite{Perkinson} meeting this criteria are readily seen to be so called strict Cayley polytopes which in turn corresponds to projective fiber bundles (see Definition \ref{caydef} and Proposition \ref{cayfiber} below).

\end{proof}
Theorem \ref{perk} will also follow independently  from our main result, see Proposition \ref{lowdim}. 

\begin{definition}\label{caydef}
Let $P_0,\dots,P_r\subset \R^k$ be polytopes. We define 
\[
[P_0*\cdots *P_r]^s:=\conv\{(P_0\times 0)\cup (P_1\times se_1)\cup \cdots \cup(P_r\times se_r)\}\subset \R^k\times \R^r 
\]
where $e_1,\dots,e_r$ is the standard basis for $\R^r$. A polytope $P\subseteq \R^n$ is called a \emph{Cayley polytope of order} $s$ \emph{and length }$r+1$ if there exist some lower dimensional polytopes $P_0,\dots, P_r$ such that $P\cong [P_0*\cdots*P_r]^s$. If $P_0,\dots, P_r$ can be taken to be normally equivalent, i.e. to have the same normal fan $\Sigma$, then $P$ is called a \emph{strict Cayley polytope}  and we write $\Cay^s_\Sigma(P_0,\dots,P_r)$ for $[P_0*\cdots*P_r]^s$.
\begin{figure}[H]
\includegraphics[scale=0.6]{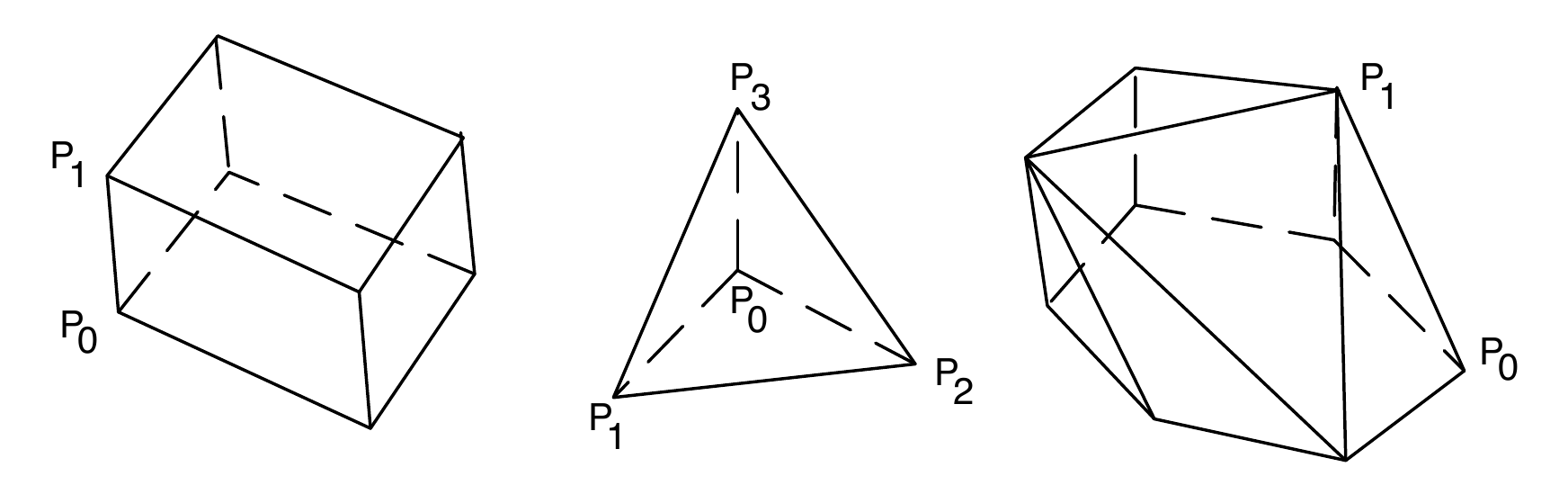}
\caption{Three Cayley polytopes in $\R^3$, two of which are strict.}
\end{figure}
\end{definition}

As previously claimed strict Cayley polytopes correspond to projective fiber bundles. The precise statement is as follows.
\begin{prop}[\cite{Casagrande},\cite{Dickenstein}]\label{cayfiber}
Let $P$ be the smooth polytope associated to a complete embedding of a toric variety $X_P$. Then $P\cong \Cay_\Sigma^s(P_0,\dots,P_k)$ if and only if $X_P$ is a $\pn k$-fiber bundle over $X_\Sigma$.
\end{prop}

The algebro geometric interpretation of having a smooth toric embedding associated to a generalized Cayley polytope $[P_0*\cdots *P_r]^s$ is not as clean as Proposition \ref{cayfiber}. However as noted in the introduction there exist a birational morphism $\pi:X'\to X$, where $X'$ is a projective fiber bundle with fiber $\pn r$. Here $X'=\pn {} (L_0\oplus \cdots\oplus L_r)$, where the $L_i$ are line bundles on the toric variety associated to (the inner-normal fan of) the Minkowski sum $P_0+\cdots+P_r$. Moreover $\pi^*\L|_F=\O_F(s)$ for all fibers $F\cong \pn r$ (see \cite{Dickenstein} for details).

\subsection{Seshadri constants on toric varieties}
Seshadri constants measure the local positivity of a line bundle on a smooth projective variety and as we soon shall see relate both to $k$-jet spannedness and to Cayley polytopes. As a motivation to the definition of Seshadri constants recall the Seshadri condition for ampleness
 \begin{theorem}[\cite{AmpleHartshorne}]\label{crit}
Let $\L$ be a line bundle on a smooth projective variety $X$. Then $\L$ is ample if and only if there exist an $\epsilon\in \R_{>0}$ such that for any point $x\in X$ and for every curve $C$ passing through $x$ it holds that $\L\cdot C\ge \epsilon\cdot  m_x(C)$, where $m_x(C)$ is the multiplicity of $C$ at $x$.
\end{theorem}
It is natural to ask for a lower bound on $\epsilon$ in Theorem \ref{crit}. Such questions lead to the following definition due to Demailly \cite{Demailly}:
\begin{definition}
Let $\L$ be a nef line bundle on a smooth projective variety $X$. The Seshadri constant of $\L$ at a point $x\in X$ is the number
\[
\epsilon(X,\L;x):=\inf_{C\subseteq X} \frac{\L\cdot C}{m_x(C)}.
\]
Here the infimum is taken over all irreducible curves $C$ passing through $x$ and $m_x(C)$ is the  multiplicity of $C$ at $x$.
\end{definition}
We remark that the two versions of Seshadri's condition for ampleness stated as Theorem \ref{crit}  and \ref{seshadriample} are obviously equivalent. In the case of an ample line bundle $\L$, a motivation for studying Seshadri constants is that they measure the jet separation of $|t\L|$ as $t$ increases. This is explained by the following Theorem. Recall that for a polarized variety $(X,\L)$ we let $s(\L,x)$ denote the largest $k\in \N$ such that $\L$ is $k$-jet spanned at $x\in X$.  

\begin{theorem}\cite{Demailly}*{Thm 6.4}\label{Demlimit}
Let $(X,\L)$ be a smooth polarized variety. Then
\[
\epsilon(X,\L;x)=\lim_{t\to \infty} \frac{s(t\L,x)}{t}
\]
for all $x\in X$.
\end{theorem}

For a toric variety $X$ we have the following Proposition due to Di Rocco which for a fixpoint $x\in X$ relate $\epsilon(X,\L;x)$ and $s(\L,x)$ more directly then Theorem \ref{Demlimit},
\begin{prop}[\cite{primer}*{4.2.2}]\label{epstok}
Let $(X,\L)$ be a smooth polarized toric variety and $x\in X$ be a torus fixpoint, then $\epsilon(X,\L;x)=s(\L,x)$.
\end{prop}
Corollary \ref{equivTHM} provide a new connection between Seshadri constants and $k$-jet spannedness, which we will prove in the final section of this paper.

\subsection{Combinatorial bounds for Seshadri constants}
In \cite{degen} Ito gives lower and upper bounds for $\epsilon(X_P,\L_P;x)$ at a point in any orbit of the torus action. In particular he deals with general points since these lie in the big orbit. For the general point these bounds can be expressed in terms of two combinatorial invariants $s_1(P)$ and $s_2(P)$. In order to define $s_1(P)$ and $s_2(P)$ we need the following setup:
Let $P$ be a rational polytope in $M_\R$ and let 
\[
\phi:=\{\pi:M_\R\to (\Z)_\R: \pi \text{ is a lattice projection } \}. 
\]
Then for a fixed $\pi \in \phi$ we can choose a decomposition $M= \ker \pi|_M \oplus \Z$, and get the following short exact sequence of free abelian groups:
\[
\xymatrix{ 
0\ar[r] &\ker \pi|_M \ar[r] &M \ar[r]^{\pi |_M} & \Z \ar[r] &0}.
\]
\begin{definition}
Let $\pi$ and $\phi$ be as above. If the rank of $M$ is 1, we define $s_1(P):=|P|_M$ and inductively if the rank of $M$ is $r$, then 
\[
s_1(P):=s_1^M(P):=\sup_{\pi \in \phi} \min\{|\pi(P)|_\Z,\sup_{u\in \Q}s_1^{\ker \pi|_M}(\pi^{-1}(u)\cap P)\}.
\]
Moreover let
\[
s_2(P):=\inf_{\pi\in \phi}|\pi(P)|_\Z.
\]
\end{definition}
\begin{figure}[H]
\includegraphics[scale=0.7]{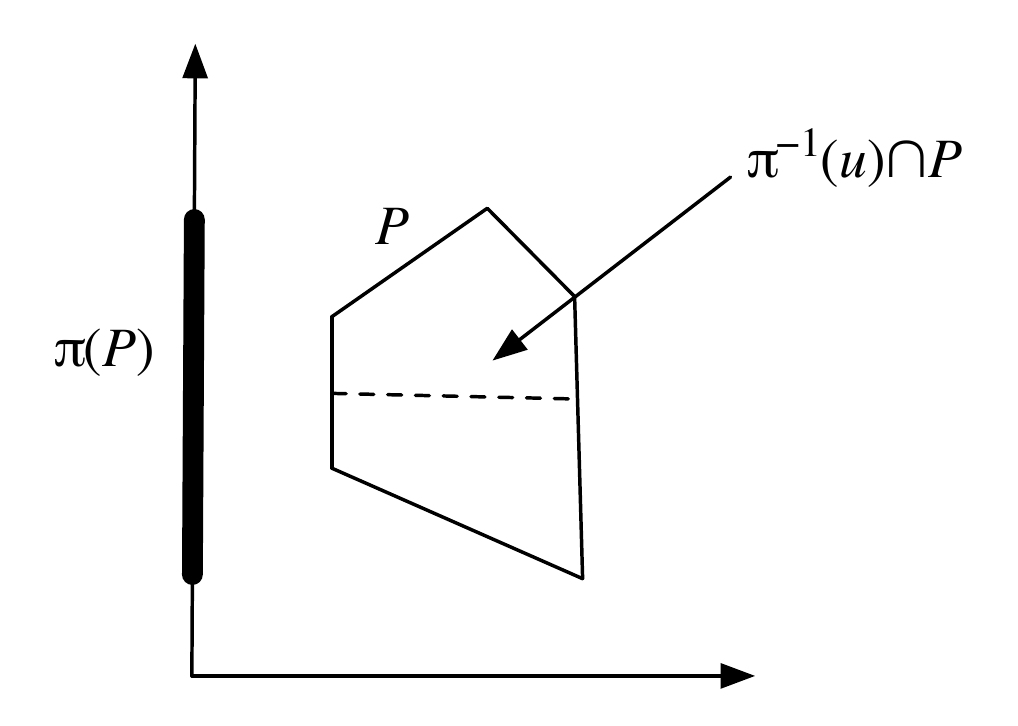}
\end{figure}

We note that $s_2(P)$ is the lattice width of $P$. It is also easy to see that $s_1(P)$ and $s_2(P)$ are well-defined (i.e. do not depend on the choice of a decomposition $M=\ker\pi_M\oplus \Z$) and are invariant under lattice isomorphisms (see \cite{degen}). 

\begin{theorem}[\cite{degen}]
Let $P\subset M_\R$ be a rational convex $n$-dimensional polytope. At the general point $x\in X_P$, it holds that
\[
s_1(P)\le \epsilon(X_P,\L_P;x)\le s_2(P).
\] 
\end{theorem}

We will compute the Seshadri constant at a general point in a few examples using the following lemma.
\begin{lemma}\label{sicomp}
Let $P$ and $Q$ be $n$-dimensional polytopes in $M_\R$
\begin{enumerate}
\item{If $Q\subseteq P$ then $s_1(Q)\le s_1(P)$ and $s_2(Q)\le s_2(P)$}
\item{Let $Q=\conv\{p+c_1^+\hat{e}_1,p-c_1^-\hat{e}_1,p+c_2^+\hat{e}_2,p-c_2^-\hat{e}_2,\dots, p+c_n^+\hat{e}_n,p-c_n^-\hat{e}_n\}$, where $p\in M_\R$, $c_1^\pm,\dots,c_n^\pm\in \R^+$ and $\hat{e}_1,\dots,\hat{e}_n$ is a basis for $M$. Then \[s_1(Q)=\epsilon(X_Q,\L_Q,x)=s_2(Q)=\min\{(c_1^+-c_1^-),\dots,(c_n^+-c_n^-)\},\] where $x\in X_Q$ is a general point.}
\end{enumerate}
\end{lemma}
\begin{proof}
This is easily proved using induction on dimension and part \emph{(1)} appears in \cite{degen}.
\end{proof}
Note that standard simplices, boxes and cross-polytopes are special cases of part \emph{(2)} in Lemma \ref{sicomp} 
\begin{example}\label{delPezzo}
The following polygon $P$ corresponds to a closed embedding of the Del Pezzo surface $X$ of degree 6 in $\pn 6$.
\begin{figure}[H]
\includegraphics[scale=0.4]{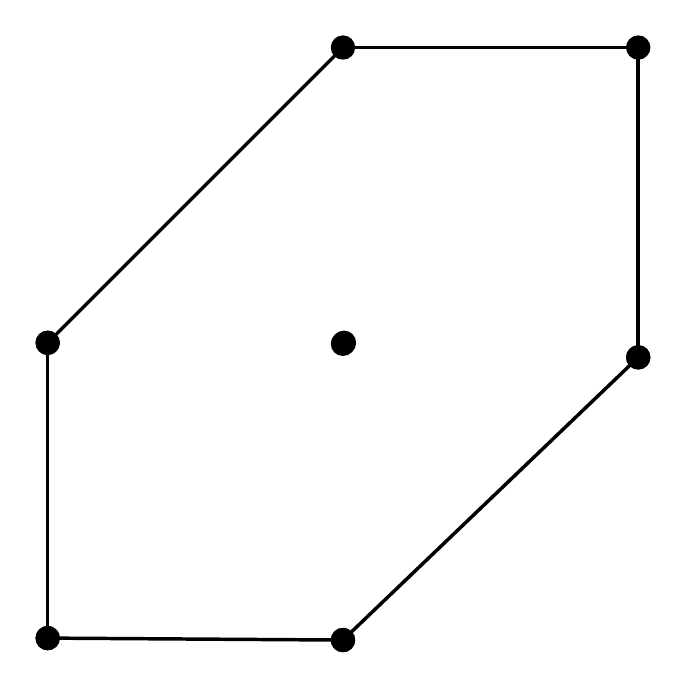}
\caption{A Del Pezzo surface in $\pn 6$}
\end{figure}
Here $\pi:X\to \pn 2$ is the blow-up of $\pn 2$ in 3 points embedded by $-K_X=\pi^*\O_{\pn 2}(3)-E_1-E_2-E_3$ where $E_1,E_2$ and $E_3$ are the exceptional divisors. The polygon corresponds to the restriction of this embedding to an affine chart. Projecting onto a coordinate axis we see that $s_2(P)\le 2$. On the other hand $P$ contains the polytope $Q=\conv\{(1,1)\pm \hat{e}_1,(1,1)\pm \hat{e}_2)\}$. Thus $s_1(P)\ge 2$ and $s_2(P)\ge 2$ by Lemma \ref{sicomp}. Hence $s_1(P)=s_2(P)=\epsilon(X_P,\L_P;x)=2$ at a general point. 
\end{example}

Note that in Example \ref{delPezzo} $\epsilon(X_P,\L_P;x(v))=1$ at every fixpoint $x(v)$ by Proposition \ref{epstok}, since every edge of $P$ has length 1.

\begin{example}\label{cayseps}
Let $P=[P_0*\cdots P_r]^k\subset \R^s\times \R^r$ and assume that every edge of $P$ has length at least $k$. Note that $P$ contains the standard simplex $k\Delta_{s+r}$, so by Lemma \ref{sicomp} we have $s_1([P_0*\cdots *P_r]^k)\ge k$. On the other hand if we consider a projection onto any coordinate axis in $\R^r$ we see that $s_2([P_0*\cdots *P_r]^k)\le k$, hence \[
s_1([P_0*\cdots *P_r]^k)=\epsilon(X_{[P_0*\cdots *P_r]^k},\L_{[P_0*\cdots *P_r]^k};x)=s_2([P_0*\cdots *P_r]^s)=k.
\]
for a general point $x\in X_P$.
\end{example}

In particular Example \ref{cayseps} tell us that if $P\cong[P_0*P_1]^1$, then $\epsilon(X,\L;x)=1$ at the general point. The following theorem due to Ito gives the converse. 

\begin{theorem}[\cite{algebro}]\label{itochar}
Let $P\subset M_\R$ be a full dimensional polytope. Then the following are equivalent
\begin{enumerate}
\item{$P\cong [P_0*P_1]^1$} 
\item{$\epsilon(X_P,\L_P,x)=1$ at a very general point $x\in X_P$.}
\item{The polarized variety $(X_P,\L_P)$ is covered by lines, i.e. for every $x\in X_p$ there exist a subvariety $Z\subset X$ such that $x\in Z$ and  $(Z,\L_P|_Z)\cong (\pn 1,\O(1))$.}
\end{enumerate}
\end{theorem}
Recall that $x\in X$ is a \emph{very general point} if $x$ lies in the complement of a countable union of Zariski closed proper subsets of $X$. Thus a very general point is a general point but the converse does not necessarily hold. Note that Example \ref{delPezzo} is an example of a polytope which is not Cayley but is such that $\epsilon(X_P,\L_P;x)=2$ at a general point. Hence a direct generalization of Ito's characterization involving Cayley polytopes of higher order is false. However Example \ref{cayseps} tells us that if  $P\cong [P_0*P_1]^k$ and every edge of $P$ has length at least $k$, then $\epsilon(X_P,\L_P;x)=k$ at the general point. Corollary \ref{equivTHM} shows that the converse hold and that the assumption is equivalent to requiring that $\epsilon(X_P,\L_P;x)=k$ at every point.

\section{Osculating spaces and Cayley polytopes }
In this section we will prove Theorem \ref{theTHM}. To this end we first need to prove a few technical lemmas.

\begin{lemma}\label{dilation}
Assume that $P$ is a smooth lattice polytope with every edge of the same length $k$. Then $\frac{1}{k}P$ is also a smooth lattice polytope with every edge of length 1. 
\end{lemma}

\begin{proof}
It is clear that $P$ and $\frac{1}{k}P$ are normally equivalent since the edge-direction at every vertex coincides. 
Let $(X,\L_P)$ be the smooth polarized toric variety corresponding to $P$ and $(X,\L_{(1/k)P})$ be the smooth $\Q$-polarized toric variety corresponding to $\frac{1}{k}P$. That every edge of $P$ has length exactly $k$ means that $\L_P\cdot C=k$ for every torus-invariant curve $C$ of $X$. Thus for all such $C$ we have $\L_{(1/k)P}\cdot C=(1/k)\L_P \cdot C=1$. It follows that $\L_{(1/k)P}$ restricted to every torus-invariant curve $C$ has degree $1$, i.e. $(X|_C,\L_{(1/k)P}|_C)\cong (\pn 1,\O_{\pn 1}(1))$, which proves the claim.
\end{proof}




To increase readability we make the following two notational definitions.
\begin{definition}
Let $P$ be a smooth polytope, we say that $P$ is \emph{canonically positioned} if $P$ has a vertex at the origin and an edge along every axis in the coordinate direction.
\end{definition}

\begin{definition}
Let $\hat{e}_1,\dots,\hat{e}_n$ be a basis for $\R^n=\R^{n-r}\times \R^r$. For a Cayley polytope $P:=[P_0*\dots*P_r]^k=\conv\{(P_0\times 0)\cup (P_1\times k\hat{e}_{n-r})\cup \cdots\cup (P_r\times k\hat{e}_n)\} \subset \R^{n-r}\times \R^r$ we say that $P$ has its \emph{altitudes} in the directions of $\hat{e}_{n-r}$ to $\hat{e}_n$.
\end{definition}

\begin{lemma}\label{triplets}
Let $P$ be a smooth canonically positioned polytope such that every facet of $P$ is a Cayley polytope of order $k$. Consider the set of triplets $(F_j,F_l,i)$ such that $i\in \{1,\dots,n\}$ while $F_j$ and $F_l$ are two distinct facets of $P$ contained in coordinate hyperplanes $H_{x_j=0}$ and $H_{x_l=0}$, with $i\ne j$ and $i\ne l$. If there exist a triplet $(F_j,F_l,i)$ such that $F_j$ and $F_l$ both have an altitude in direction $\hat{e}_i$, then $P\subseteq H_{x_i=0}^+\cap H_{x_i=k}^-$.
\end{lemma}

\begin{proof}
Let $(F_1,F_2,i)$ be the stated triplet and consider the vertex $v$ of $P$ at $k\hat{e}_i$. Since $P$ is simple there are exactly $n$ edges through $v$. Since $F_1$ and $F_2$ are simple $(n-1)$-polytopes $n-1$ of the edges through $v$ lies in each of $F_1$ and $F_2$. Because $F_1$ and $F_2$ are distinct there must be at least one edge through $v$ which lies in $F_1$ and not in $F_2$ and conversely. Thus every edge through $v$ lies in $F_1$ or $F_2$. Since $F_1$ and $F_2$ have their altitude in the direction of $\hat{e}_i$ the edge-directions through $v$ all have a non-positive component in the $\hat{e}_i$ direction. From the fact that all 2-faces of $P$ are smooth, it then follows that the edge directions through $v$ are $-\hat{e}_i$ or of the form $l\hat{e}_i+\hat{e_j}$ where $ -k\le l \le 0$ and $i\ne j$. Thus the equation for the hyperplane $H$ through $v$ and the first lattice points along the edges with edge-direction $l\hat{e}_i+\hat{e}_j$ can be written as $f(x)=x_i/k+\sum_{j=1\\j\ne i}^n a_jx_j-1=0$, where $0\le a_j\le 1$ for all $j\in \{1,\dots,n\}$.  As a consequence if there exist a $p\in P$ such that $\langle p,\hat{e}_i\rangle >k$ then $f(p)>0$, since $P$ is contained in the first orthant. However $H$ is a supporting hyperplane of $P$ and $P\subset H^-$, which gives a contradiction.
\end{proof}

\begin{lemma}\label{contcay}
Let $P\subset \R^n$ be an $n$-polytope such that every edge of $P$ has length at least $k$. If $P\subseteq H_{x_i=0}^+\cap H_{x_i=k}^-$, then there exist lower dimensional polytopes $P_0$ and $P_1$ such that $P=[P_0*P_1]^k$ with altitude in the direction of $\hat{e}_i$.
\end{lemma}

\begin{proof}
By the definition of Cayley polytopes it is enough to show that every vertex of $P$ lies in either the hyperplane $H_{x_i=0}$ or the hyperplane $H_{x_i=k}$. To this end let $v$ be a vertex of $P$ and assume that $v\not\in H_{x_i=0}$ and $v\not\in H_{x_i=k}$. If $\hat{e}$ is the direction of an edge through $v$ then the $i$:th component of $\hat{e}$ must be be identically zero, because $P\subset H_{x_i=0}^+\cap H_{x_i=k}^-$ and the edge with direction $\hat{e}$ has length at least $k$. Hence $v$ and every vertex sharing an edge with $v$ lie in a common hyperplane. Because the graph of $P$ is connected this implies that every vertex of $P$ lies in the same hyperplane, which contradicts that $P$ is full dimensional. 

\end{proof}

\begin{lemma}\label{twoheights}
Let $P\cong [P_0*P_1]^1$ be smooth and canonically positioned with its altitude in the direction of $\hat{e}_1$. If  $e=\{\hat{e}_1,\hat{e}_2\}$  is an edge of $P$, then $\codim(P_1)>1$. Moreover there exist lower-dimensional polytopes $P_0'$,  $P_1'$ and $P_2'=P_1$ such that $P\cong [P_0'*P_1'*P_2']^1$. 
\end{lemma}

\begin{proof}
Let $v$ be a vertex of $P$ in the hyperplane $H_{x_1=k}$. Because $P_1=P\cap H_{x_1=1}$ and $P_1$ is smooth we have that there are $\dim(P_1)$ many edges through $v$, which are contained in the hyperplane $H_{x_1=1}$. On the other hand there are exactly $n$ edges of $P$ through the vertex $v$, where $e$ and the edge $\{0,\hat{e}_1\}$ are edges through $v$ that are not contained in $H_{x_1=1}$. Thus $\codim(P_1)>1$. Let $H$ be the supporting hyperplane of $P$ passing through $v$ as well as every neighbouring vertex of $v$ except the origin. A linear algebra computation, completely analogous to the one in the proof of Lemma \ref{triplets}, implies that $H$ is given by an equation of the form $H(x):=x_1+x_2+\sum_{i\in I}x_i-1=0$, where $I$ is some subset of $\{1,\dots,n\}$. This implies that $P\subset H_{x_1=1}^-$ and $P\subset H_{x_2=1}^-$ since if $p\in P$ and $\langle p,\hat{e}_i\rangle >1$, where $i=1$ or $2$, then $H(p)>0$, because $P$ is contained in the first orthant. Thus $P\cong[P_0*P_1*P_2]^1$ by Lemma \ref{contcay} applied twice.
\end{proof}

\begin{theorem}
Let $(X,\L)$ be a smooth polarized toric variety and let $P_\L$ be the polytope associated to the complete linear series $|\L|$. $\L$ is $k$-jet spanned but not $(k+1)$-jet spanned at every point $x\in X$ if and only if $P\cong [P_0*P_1]^k$ for some lower dimensional polytopes $P_0$ and $P_1$ and every edge of $P$ has lattice length at least $k$.
\end{theorem}

\begin{proof}
Assume $P=[P_0*P_1]^k\subset \R^n$ and that every edge of $P$ has length at least $k$. Note that every vertex of $P$ lies in either the face corresponding to $P_0$ or the face corresponding to $P_1$. Since every edge connecting $P_0$ and $P_1$ in $P$ has length exactly $k$ we have that $\L_P$ is $k$-jet spanned, but not $(k+1)$-jet spanned, at every fixpoint of $X_P$ by Proposition \ref{bigprop}. Again by Proposition \ref{bigprop} this implies that $\L_P$ is $k$-jet spanned at the general point. If $e_1,\dots,e_n$ is the standard basis of $\R^n$, then we can assume that $P=\conv\{(P_0\times 0)\cup (P_1\times ke_n)\}$, so the polynomial $f_p=x_n(x_n-1)(x_n-2)\cdots (x_n-k)$ is a degree $k+1$ polynomial vanishing on $P\cap \Z^n$. Thus by Proposition \ref{bigprop} $\L_P$ is not $(k+1)$-jet spanned at the general point. From Remark \ref{fitideal} we can thus conclude that $s(\L,x)=k$ for all points $x\in X$. This proves the ''if'' direction.

We will prove the converse by induction on the dimension of $X$, since the statement is obvious in dimension one. Thus assume that the Theorem holds in dimension $\le n-1$ and that $X$ has dimension $n$. The Theorem will be established by first showing that we can do two reductions. The first reduction is the following:
\begin{claim}\label{allcay}
We may assume that every facet of $P$ is a Cayley polytope.
\end{claim}

\begin{proof}[Proof of Claim \ref{allcay}]
Assume that $P$ has a facet $F$ which is not a Cayley polytope. Note that the toric embedding corresponding to $F$ is $k$-jet spanned since every edge of $F$ has length at least $k$. Hence if $f_P$ is a polynomial of degree $k+1$ vanishing on $P\cap M$ then $f_P$ restricted to the supporting hyperplane of $F$ must be identically zero, since otherwise we would get a contradiction to the Theorem in dimension  $n-1$. Now without loss of generality we may assume that $P$ is canonically positioned and that $F$ is the facet in the hyperplane $H_{x_1=0}$. Then the above assumption implies that $f_P=x_1g$ where $g$ is a degree $k$ polynomial vanishing on $P':=P\cap H_{x_1=1}^+$.

Observe that $(X',\L_{P'})$ is the blow-up of $X$ along a torus-invariant effective (Cartier) divisor $E$, embedded via $L_{P'}=\pi^*\L_P-E$, where $\pi:\tilde{X}\to X$ is the blow-up map. Thus $X'\cong X$ by the universal property of blow-ups, i.e. $P$ and $P'$ are normally equivalent.


Next we show that every edge of $P'$ has length at least $k-1$, which will lead to a proof of the claim. Clearly any edge of $P'$ which corresponds to an edge $e$ in $P$, that either has empty intersection with $H_{x_1=0}$ or intersect $H_{x_1=0}$ in a point, will have length at least $k-1$. Now consider an edge $e$ of $P$ which lies in the hyperplane $H_{x_1=0}$. Let $v$ be a vertex of $e$, then translating $P$ so that $v$ is at the origin and changing basis yields a isomorphism $\phi$ of $P$ such that:\begin{enumerate}
\item{$\phi(P)$ is positioned with the facet $F$ in the hyperplane $H_{x_1=0}$.}
\item{$\phi(P)$ is canonically positioned.}
\item{$\phi(H_{x_1=c})=H_{x_1=c}$ for all $c\in \R$.}
\end{enumerate}
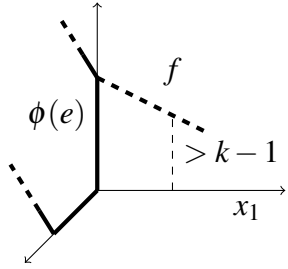
\begin{figure}[h]
\begin{tikzpicture}
\draw [->](0,0,0)--(2.5,0,0);
\draw [->](0,0,0)--(0,2.5,0);
\draw [->](0,0,0)--(0,0,2.5);
\node [below] at (2,0,0) {$x_1$};
\draw [dashed](1,0,0)--(1,1,0);
\node [right]at (1,0.5,0) {$>k-1$};
\node [left] at (0,1,0) {$\phi(e)$};
\draw [ultra thick] (0,0,0)--(0,1.5,0);
\draw [dashed, ultra thick] (0,1.5,0)--(1.5,0.75,0);
\draw [ultra thick] (0,0,0)--(0,0,1.5);
\draw [ultra thick] (0,0,1.5)--(0,0.5,2);
\draw [dashed, ultra thick] (0,0.5,2)--(0,1.5,3);
\draw [ultra thick] (0,1.5,0)--(0,2,0.5);
\draw [dashed, ultra thick] (0,2,0.5)--(0,2.75,1.25);
\node [above] at (1,1.25,0) {$f$};
\end{tikzpicture}
\caption{Figure for the proof of Claim \ref{allcay}}
\end{figure}
By construction $\phi(P')=\phi(P\cap H_{x_1=1}^+)=\phi(P)\cap H_{x_1=1}^+$. Assume $\phi(e)$ lies in the $x_1x_2$-plane. By smoothness there exist an edge $f$ of $\phi(P)$ in the $x_1x_2$-plane passing through the vertex of $e$ that is not the origin. Because $\phi(e)$ and $f$ both have length at least $k$ and $\phi(P)$ is contained in the first orthant the edge corresponding to $\phi(e)$ in $\phi(P')$ has length at least $k-1$. Because edge lengths are invariant under isomorphisms, we conclude that $P$ and $P'$ are normally equivalent with $P'$ satisfying the assumptions of the Theorem with $k-1$ in place of $k$. By iterating the procedure we get that $f_P=x_1(x_1-1)\cdots(x_1-k-1)h$ where $h$ is a degree one polynomial vanishing on $P\cap H_{x_1=k-1}^+$.  Now there are two cases: Either $\dim(P)\cap H_{x_1=k})=n-1$ in which case we can iterate the procedure once more to get that $f_p=c\prod_{i=0}^k(x_1-i)$, where $c$ is a constant. Or $\dim(P)\cap H_{x_1=k})<n-1$ in which case the hyperplane $H_{x_1=k}$ cuts out a face of $P$ so that $P\subset H_{x_1=k}^-$. In both cases $P\subset H_{x_1=0}^+\cap H_{x_1=k}^-$, so Lemma \ref{contcay} implies that $P\cong[P_0*P_1]^k$. This proves Claim \ref{allcay}.
\end{proof}

Our second claim is the following:
\begin{claim}\label{step2}
Without loss of generality we can assume that every edge of $P$ has length exactly $k$ and no 2-face of $P$ is isomorphic to $k\Delta_2$.
\end{claim}

\begin{proof}[Proof of Claim \ref{step2}]
By Claim \ref{allcay} we can assume that every facet of $P$ is a Cayley polytope. Assume moreover that $P$ has an edge $e$ of length strictly larger then $k$. Let $P$ be canonically positioned with $e$ along the $x_1$-axis. Then the facets of $P$ in the coordinate hyperplanes must have their altitudes in the direction of a coordinate axis. Because none of these $n$ facets can have its altitude on the $x_1$-axis there are  $n-1$ possible choices for the the direction of the altitudes. Thus by the pigeonhole principle at least two facets have their altitude in the same direction. By Lemmas \ref{triplets} and \ref{contcay} we conclude that $P\cong[P_0*P_1]^k$ in this case. Assume instead that $P$ has a 2-face $F$ which is isomorphic to $k\Delta_2$. Note that by the above and Lemma \ref{dilation} we may assume that $k=1$, since $s(\L,x)=s(k\L',x)=k$ at every point implies that $s(\L',x)=1$ at every point. Let $P$ be cannoncially positioned with $F$ contained in the linear space given by $x_3=x_4=\dots =x_n=0$. By Lemmas \ref{triplets} and \ref{contcay} we may moreover assume that every coordinate direction is the altitude of exactly one facet contained in a coordinate hyperplane. There are two cases: either there exist a facet which has its altitude in the direction of the $x_1$- or $x_2$-axis but does not contain $F$, or the facets having their altitude in the direction of $x_1$ and $x_2$ both contain $F$. In the first case let $H_1$ be the facet having its altitude in the direction of $x_1$. Then, up to interchanging the role of $x_1$ and $x_2$, all edge-directions through the vertex $v$ at $\hat{e}_1$ are either the edge directions of $v$ in $H_1$ or $-\hat{e}_1+\hat{e}_2$. Now the exact same argument as in the proof of Lemma \ref{triplets}, together with Lemma \ref{contcay} implies that $P\cong[P_0*P_1]^k$. In the second case let $H_1$ be the facet with its altitude in the direction of $x_1$. Then Lemma \ref{twoheights} implies that $H_1$ has at least two altitudes. Thus in this case $P\cong [P_0*P_1]^k$ follows from the pigeonhole principle together with Lemma \ref{triplets} and \ref{contcay}. This proves Claim \ref{step2}.
\end{proof}

By Claim \ref{step2} we may assume  that every edge of $P$ has length $k$ and by Lemma \ref{dilation} that $k=1$. If we assume that $P$ is canonically positioned, then the second part of Claim \ref{step2} implies that we also without loss of generality can assume that $0,e_i,e_i+e_j\in P$ for all $i,j\in \{1,\dots,n\}$. Now saying that $(X,\L)$ is not 2-jet spanned at the general point implies that there exist a degree 2 polynomial $f_P(x)$ vanishing on $P\cap M$. The fact that $f_P(x)$ vanish on $0,\hat{e}_i,\hat{e}_i+\hat{e}_j$ for all $i,j\in \{1,\dots,n\}$ implies that $f_P(x)=\sum_{i=0}^nc_ix_i(x_i-1)$ for some coefficients  $c_1,\dots,c_n\in \C$. We will show that either $P\cong[P_0*P_1]^1$ or there is an additional family of lattice points in $P$ which implies that $c_i=0$ for all $i\in \{1,\dots,n\}$. To this end assume that there exist an $i$ such that $2\hat{e}_i+\hat{e}_j\not\in P$ for all $j\in \{1,\dots,n\}$ and note that  every 2-face in a linear space determined by the equations $x_1=x_2=x_3=\dots =\check{x_i}=\dots=\check{x_j}=\dots =x_n=0$ is smooth. Thus the edge-directions through the vertex $v$ at $\hat{e}_i$ are of the form $-\hat{e}_i$, $\hat{e}_j$ or $-\hat{e}_i+\hat{e}_j$ for $i\ne j$. This implies that the supporting hyperplane of $P$ through $v$ and all of its neighbouring vertices except 0 is determined by an equation of the form $H(x)=x_i+\sum_{j\in I} x_j -1=0$, where $I$ is some subset of $\{1,\dots,\check{i},\dots,n\}$. Thus $P\subset H_{x_i=1}^-$ since if $p\in P$ and $\langle p,\hat{e}_i\rangle>1$, then $H(p)>0$ because $P$ is contained in the first orthant. Hence if there exist an $i$ such that $2\hat{e}_i+\hat{e}_j\not\in P$ for all $j\in \{1,\dots,n\}$, then $P\cong [P_0*P_1]^1$. On the other hand if for some $i$ there exist a $j\in \{1,\dots,n\}$ such that $2\hat{e}_i+\hat{e}_j\in P$, then $f_P(x)$ has to vanish at that lattice point, which implies that $c_i=0$. Thus either $P\cong [P_0*P_1^1$ or $f_P(x)=0$, which is a contradiction.
\end{proof}

\begin{corollary}
Let $(X,\L)$ be a smooth polarized toric variety, let $P_\L$ be the corresponding smooth polytope and let $k\in \N$. Then the following statements are equivalent:
\begin{enumerate}[i)]
\item{$s(\L,x)=k$ at every point $x\in X$.}
\item{$s(\L,x)=k$ at the fixpoints and at the general point.}
\item{$\epsilon(X,\L;x)=k$ at every point $x\in X$.}
\item{$\epsilon(X,\L;x)=k$ at the fixpoints and at the general point.}
\item{$P_\L\cong[P_0*P_1]^k$ for some lower dimensional polytopes $P_0$ and $P_1$ and every edge of $P$ has length at least $k$.}
\end{enumerate}
\end{corollary}
\begin{proof}
That \emph{i)} and \emph{ii)} are equivalent is part of Remark \ref{fitideal}. Next \emph{iv)} implies \emph{ii)} by Theorem \ref{Demlimit}, Proposition \ref{epstok} and Proposition \ref{bigprop} since $t\L$ is $tk$-jet spanned at $x\in X$ if $\L$ is $k$-jet spanned at $x$. Now Theorem \ref{theTHM} shows that \emph{ii)} and \emph{v)} are equivalent. Moreover \emph{v)} implies {iv)} by Example \ref{cayseps} and Proposition \ref{epstok}. Finally it is clear that \emph{iii)}  implies \emph{iv)}, so all that remains to prove is that \emph{iv)} implies \emph{iii)}. To this end assume \emph{iv)} and note that by what we just proved $s(\L,x)=k$ at every point i.e. $\epsilon(X,\L;x)\ge k$ at every point. By Remark \ref{fitideal} $s(t\L,x)\le s(t\L,1)=tk$ for all points $x\in X$ and $t\in \N$. Thus dividing by $t$ and taking the limit as $t\to \infty$, we get that $\epsilon(X,\L;x)\le \epsilon(X,\L;1)=k$ at every point $x\in X$ by Theorem \ref{Demlimit}. This establishes the last implication needed. 
\end{proof}

Recall that a consequence of the polytope $P$ decomposing as a Cayley sum $[P_0*\cdots *P_1]^k$ is that there exist a birational morphism $\pi:X'\to X$, where $X'$ is a projective fiber bundle, as previously noted. Finally we prove that our classification coincide with the one of Perkinson by showing that if $P$ is a smooth Cayley polytope of dimension at most 3, then $P$ is indeed strict. 

\begin{prop}\label{lowdim}
Let $P$ is a smooth polytope of dimension at most 3. If $P\cong [P_0*P_1]^k$ for some $k\in \N$ and lower dimensional polytopes $P_0$ and $P_1$, then $P$ is a strict Cayley polytope of order $k$. 
\end{prop}

\begin{proof}
The statement is obvious in dimension less than three. Moreover if, in dimension 3, either $P_0$ or $P_1$ is a point or both $P_0$ and $P_1$ are line segment, then $P$ has 4 vertices and equally many facets, thus $P\cong k\Delta_3$ by Mabuchi's Theorem \cite{OdaT}*{Thm. 7.1}. Next assume that $P_0$ is a polygon while $P_1$ is a line segment. Because every vertex in $P_0$ share an edge with a vertex in $P_1$ we see that $P$ has 6 vertices i.e. $P_0$ is a quadrilateral. Because $P$ is simple this implies that $P$ has $5$ facets by Euler's formula. Thus $\Pic(X)\cong \Z^2$, so from Kleinschmidt's Classification Theorem (\cite{Kleinschmidt}, \cite{Cox}*{p.341}), we see that $(X,\L_P)$ is again a projective fiber bundle. 

The only remaining case is if $P_0$ and $P_1$ are both polygons. Note first that through every vertex coming from $P_0$ there is exactly one edge having its other endpoint at a vertex coming from $P_1$ and similarly with $P_0$ and $P_1$ interchanged. This has two consequences: Firstly $P_0$ and $P_1$ have the same number of vertices. Secondly no facet of $P$, except possibly $P_0$ or $P_1$, is a simplex. Now let $H_{P_i}$ be the hyperplane containing the facet $P_i$. If $H$ is a supporting hyperplane of $P$ and $H\ne H_{P_i}$ $i=0,1$, then, by the above, $H$ contains one edge $e_i$ of $P$ coming from $P_i$ for $i=0,1$. However if $L_i$ is the line through $e_i$, then $L_i=H\cap H_{P_i}$ for $i=0,1$. Thus we can conclude that if $H_{P_0}$ and $H_{P_1}$ are parallel then $L_0$ and $L_1$ are parallel in $H$ and $e_0$ and $e_1$ are parallel in $H$. Thus every edge in $P_0$ has a parallel edge  in $P_1$ and these two  edges lie in the same facet of $P$. This determines all edge-directions in $P_0$ and $P_1$ and we conclude that $P_0$ and $P_1$ are normally equivalent.
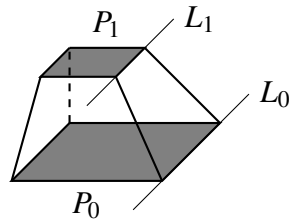
\begin{figure}[h!]
\begin{tikzpicture}
\draw [thick,fill=gray](0,0,0)--(2,0,0)--(2,0,2)--(0,0,2)--(0,0,0);
\draw [thick, fill=gray](0,1,0)--(1,1,0)--(1,1,1)--(0,1,1)--(0,1,0);
\draw (2,0,-1)--(2,0,3);
\draw (1,1,-1)--(1,1,2);
\node [right] at (2,0,-1) {$L_0$};
\node [right] at (1,1,-1) {$L_1$};
\draw [thick] (2,0,0)--(1,1,0);
\draw [thick] (2,0,2)--(1,1,1);
\draw [thick] (0,0,2)--(0,1,1);
\draw [dashed, thick](0,0,0)--(0,1,0);
\node [above] at (0.5,1,0) {$P_1$};
\node [below] at (1,0,2) {$P_0$};
\end{tikzpicture}
\caption{Figure for the proof of Proposition \ref{lowdim}}
\end{figure}
\end{proof}

\section{Acknowledgements}
I would like to thank my advisor Sandra Di Rocco for introducing me to the problem and for guidance along the way. I am also in debt to Christian Haase, Benjamin Nill and Erik Aas for valuable discussions and input on the problem. Finally I would like to thank the Department of Mathematics at KTH in Stockholm and Vetenskapsrådet in Sweden, for their financial support.
\bibliographystyle{ieee}
\bibliography{seshadriarticle}
\end{document}